\documentclass[12pt]{amsart}
\usepackage{amsmath, amssymb}
\usepackage{newtxtext, newtxmath}
\usepackage{mathtools, a4wide, bm, amsfonts}
\usepackage{paralist}

\usepackage[T1]{fontenc}

\usepackage[colorlinks=true,linkcolor=blue,citecolor=blue,pdfpagelabels=false]{hyperref}
\numberwithin{equation}{section}

\newenvironment{psm}
  {\left(\begin{smallmatrix}}
  {\end{smallmatrix}\right)}

\theoremstyle{plain}
\newtheorem{thm}{Theorem}[section]
\newtheorem{lem}[thm]{Lemma}
\newtheorem{prop}[thm]{Proposition}
\newtheorem{cor}[thm]{Corollary}
\newcommand{\thmref}[1]{Theorem~\ref{#1}}
\newcommand{\lemref}[1]{Lemma~\ref{#1}}
\newcommand{\propref}[1]{Proposition~\ref{#1}}

\theoremstyle{definition}
\newtheorem{rmk}[thm]{Remark}

\newtheorem{conj}[thm]{Conjecture}
\newtheorem{exam}[thm]{Example}
\newcommand{\rmkref}[1]{Remark~\ref{#1}}

\newcommand{\conjref}[1]{Conjecture~\ref{#1}}

\newcommand{\mf}{\mathbf}
\newcommand{\q}{\quad}
\newcommand{\qq}{\qquad}

\newcommand{\mc}{\mathcal}
\newcommand{\mk}{\mathfrak}
\newcommand{\mrm}{\mathrm}
\newcommand{\sltwo}{\mrm{SL}(2, \mf Z)}
\newcommand{\spn}{\mrm{Sp}(n,\mf Z)}

\newcommand{\glnz}{\mrm{GL}(n, \mf Z)}

\newcommand{\mnk}{M^n_k(\Gamma_1(N))}
\newcommand{\snk}{S^n_{k}(\Gamma_1(N))}

\newcommand{\mN}{M_{k}(\Gamma_1(N))}
\newcommand{\mNok}{M_{k}(\Gamma_1(N))(\mc O_K)}

\newcommand{\mm}{M_{\kappa}(\Gamma_1(M))}
\newcommand{\sN}{S_{k}(\Gamma_1(N))}

\newcommand{\nl}{\not \equiv 0 \bmod \mk l}
\newcommand{\ml}{\bmod \, \mk l}
\newcommand{\oml}{\equiv 0 \bmod \mk l}
\newcommand{\bl}{\bmod \, \ell}

\newcommand{\ok}{\mc O_K}

\newcommand{\tr}{\mathrm{tr}}

\newenvironment{nouppercase}{%
  \let\uppercase\relax%
  \renewcommand{\uppercasenonmath}[1]{}}{}

\title[Fourier coefficients $\bmod \ell$]{On Fourier coefficients of elliptic modular forms $\bmod \, \ell$ with applications to Siegel modular forms}

\author{Siegfried B\"ocherer}
\address{Institut f\"ur Mathematik\\
Universit\"at Mannheim\\
68131 Mannheim (Germany).}
\email{boecherer@math.uni-mannheim.de}

\author{Soumya Das}
\address{Department of Mathematics\\ 
Indian Institute of Science\\ 
Bangalore -- 560012, India\\
and Humboldt Fellow, Universit\"{a}t Mannheim.}
\email{soumya@iisc.ac.in, sdas@mail.uni-mannheim.de}

\date{}

\begin{document}

\begin{abstract}
We study several aspects of nonvanishing Fourier coefficients of elliptic modular forms $\bl$, partially answering a question of Bella\"iche-Soundararajan concerning the asymptotic formula for the count of the number of Fourier coefficients upto $x$ which do not vanish $\bl$. We also propose a precise conjecture as a possible answer to this question.
Further, we prove several results related to the nonvanishing of arithmetically interesting (e.g., primitive or fundamental) Fourier coefficients $\bl $ of a Siegel modular form with integral algebraic Fourier coefficients provided $\ell$ is large enough. We also make some efforts to make this "largeness" of $\ell$ effective.
\end{abstract}

\begin{nouppercase}
\maketitle
\end{nouppercase}

\section{Introduction}
The aim of this article is to obtain $\bmod \, \ell$ versions of some of the nonvanishing results on the Fourier coefficients of Siegel modular forms. On the one hand, over $\mf C$ such results (cf. \cite{saha, bo-da, ad}) have played important roles in many questions on automorphic forms, and it seems interesting to investigate to what extent they hold over other rings, possibly in a quantitative fashion. As an example, in \cite{bo-da} it was proved that for any holomorphic Siegel modular form $F$ of degree $n$, there exist infinitely many inequivalent (modulo the unimodular group) half-integral matrices $T$ whose discriminants are fundamental, such that $a_F(T) \ne 0$. Such results have several applications to automorphic representations.

On the other hand, the theory of modular forms $\bmod \, \ell$ has undergone extensive development since the works of Serre, Swinnerton-Dyer. Let $f$ be an elliptic cuspidal newform of weight $k$, level $\Gamma_0(N)$ and $\mk l$ is a prime ideal in the ring of integers $\ok$ of a field $K$ which lies over the odd prime $\ell$. Serre used the Chebotarev density theorem applied to the setting of the Galois representation attached to $f$, and the Selberg-Delange method, to deduce that for such $\mk l \mid \ell$ and $f \not  \equiv 0 \bmod \mk l$,
\begin{equation} \label{serreasy}
 \# \pi(f,x) \sim c(f) x/(\log x)^{\alpha(f)}
\end{equation}
for some $c(f), \alpha(f)>0$. Here $\pi(f,x) := \{n \le x|  \, a(f,n) \not \equiv 0 \bmod \mk l \}$.
More recently, by the works of Bellaiche, Soundararajan, Green \cite{j, js, jsh}, quantitative results like \eqref{serreasy} has been extended to arbitrary modular forms, possibly with half-integral weights. 

In the first half of the paper (section~\ref{ellsec}), we show that
a simple sieving of newforms (inspired by \cite{baono} and relying essentially on multiplicity-one for $\mN$) leads to quantitative results similar to \eqref{serreasy} for \textsl{arbitrary} modular forms in $\mN$ of the \textsl{correct} order of magnitude, when $\ell$ is large enough. Actually our results hold for all $\ell$ not dividing a fixed algebraic integer in a number field, see below, and subsection~\ref{noncongana} for more details. This technique of sieving newforms has been useful in many places e.g. \cite{ad, das} and can also be adapted to count square-free integers $n$ for which $a(f,n) \nl$, see \propref{ellasy}. In general this method works whenever a space of modular forms has the multiplicity-one property, and the corresponding eigenforms (or newforms) possess the suitable properties in question.

Let us explain the results of this article in some detail. In section~\ref{ellsec}, we prove several results about the Fourier coefficients $\ml$ of modular forms in $\mN$, the mainstay being \propref{mainprop}. In particular in \thmref{emf}, we show an analogue of `old-form' theory for modular forms $\ml$ with fixed weight and level for all $\mk l$ not dividing a certain algebraic integer $\mc L$. This has an application to a result on Siegel modular forms about non-zero `primitive' Fourier coefficients $\ml$, see \thmref{ymodp}. Our method however certainly does not generally work on the bigger space of modular forms $\bl$ of level $N$ (lets denote it by $\widetilde{M}(N)$) as eg. in \cite{js}, because afterall by Jochnowitz \cite{jochno} the number of systems of eigenvalues $\bl$ for any $(\ell,6N)=1$ is finite. But for those $f \in\widetilde{M}(N)$ which are finite linear combinations of eigenforms with pairwise distinct system of eigenvalues $\ml$, the method clearly still works.

We next note several applications of \propref{mainprop}. To discuss some of these, let us introduce some notation. Let $f \in \mN$ be such that its Fourier coefficients belong to the ring of integers $\mc O_K$ of a number field $K$. Consider a basis $\{f_{1}, f_{2},......f_{s}\}$ of newforms of weight $k$ and level dividing $N$, including Eisenstein-newforms (cf. \cite{weis}). Let their Fourier expansions be written as 
\begin{equation} \label{newfc}
f_{i}(\tau)=\sum_{n=0}^{\infty} b_{i}(n)q^{n}, 
\end{equation}
normalised so that $b_i(1)=1$ for all $i$. For each pair $i \ne j$, let $m_{i,j}$ be the smallest prime coprime to $N$ such that $b_i(m_{i,j}) \neq b_j(m_{i,j}) $. For the rest of the paper, we put
\begin{align} \label{ldef}
\mc L = \mc L(\{ m_{i,j} \}, N,k) := \prod_{i \ne j} \left( b_i(m_{i,j}) - b_j(m_{i,j}) \right).
\end{align}
(Later we would use variants of $\mc L$, however.) Let $\mk l \in \ok$ be any prime lying over $\ell \in \mf Z$ such that $\mk l \nmid \mc L(\{ m_{i,j} \}) $. Note that the primes $m_{i,j}$ do not depend on $f$. In course of this paper, we will call such a set of primes (perhaps with additional conditions, see section~\ref{noncongana}) to be `admissible' for the modular form at hand.

In one of our results (cf. \propref{ellasy}) we show that the quantity $\pi(f,x)$ satisfies
\begin{equation} \label{pifintro}
\# \pi(f,x; \mk l):= \{ n \le x| a(f,n) \not \equiv 0 \ml  \} \asymp x/(\log x)^{\alpha(f)},
\end{equation}
for some $3/4 \ge \alpha(f)>0$ whenever $ f \nl$, $k \ge 1$ and $\mk l \nmid \mc L$.
These should be compared to the results in \cite{js} (which is valid for all modular forms $\ml$ on $\Gamma_1(N)$ with Fourier coefficients in $O_K$ and is an asymptotic formula), and in fact shows that in the asymptotic formula of \cite{js}, namely 
\begin{equation} \label{ellintro}
\# \pi(f,x, \ell) \sim x (\log \log x)^{h(f)}/(\log x)^{\alpha(f)};
\end{equation}
the integer $h(f)$ appearing above is actually $0$ if $f \in \mN$, provided $\mk l \nmid \mc L$ (actually a slightly stronger result holds, see \propref{hfl}). This may shed some light on the behaviour of $h(f)$, which the authors in \cite{js} comment as being rather mysterious, as $\ell$ varies. Apart from this, the point here is that our proofs are `softer', however we do not get an asymptotic formula. We make some efforts in finding a constant $\mc C$ depending only on $k,N$ such that \eqref{pifintro} holds for all $\ell > \mc C$. The reader may look at section~\ref{noncongana}. In fact it follows from \propref{hfl} that $h(f)=0$ for all $\ell > \mc C$ with suitable $\mc C$ as above, see \rmkref{hflcomm}. More generally, as an outcome of our line of thought, in \propref{conjod} we note that for those $f \in\widetilde{M}(N)$ which are finite linear combinations of eigenforms with pairwise distinct system of eigenvalues, one would have $h(f)=0$. We speculate that the converse to the previous statement is true as well and this is the content of \conjref{conjs}.

The reader may note that there is no contradiction with the examples in \cite[\S~7]{js} since e.g. the prime $\ell=3$ considered there divides $\mc L$ (at level $1$ and weight $24$). See example~\ref{delta-exam} concerning $f = \Delta^2$ ($\Delta$ is Ramanujan's Delta function) for some more clarity on this. Moreover if the level $N$ is square-free, we obtain results similar to \eqref{ellintro} for the set $\pi_{\mrm{sf}}(f,x) :=\{n \le x | n \text{ square-free, } a(f,n) \nl \}$. Several such results are collected in \propref{ellasy}. Finally let us mention that we briefly discuss an algebraic way to approach some of our results in subsection~\ref{congalg}.

In the second half of the paper (section~\ref{siegelsec}) we derive analogous results for Siegel modular forms. Let $F \in \mnk$ be a Siegel modular form with Fourier coefficients in the ring of integers $\mc O_K$ of a number field $K$. We first prove (see \thmref{ymodp}) that such an $F$ which is $\nl$ has  infinitely many $\glnz$-inequivalent `primitive' matrices $T \in \Lambda_n$ such that the Fourier coefficients $a_F(T) \nl$. This generalises a result of Yamana \cite{yam} who proved a similar result when $\ell = \infty$. The proof uses a refinement of a method (of descending to elliptic modular forms) presented in \cite{bona}, some results on "old-forms $\ml$" on elliptic modular forms (\thmref{emf}), and the existence of a Sturm's bound for the space $\mnk$, which is formulated and proved in \propref{siegel-sturm} generalising the level one result from \cite{raum}.

Then we give lower bounds on the number of $d \le x$ such that $d= \det(T)$ and that the Fourier coefficients $a_F(T) \nl$ for some $T$ (also satisfying additional arithmetic properties, see \thmref{smodp}, section~\ref{siegelsec}). The proofs are based on reduction of the question to spaces of elliptic modular forms via the Fourier-Jacobi expansions, and using either the results from \cite{j, js, jsh}; or sometimes using the lower bounds (cf. section~\ref{ellquan}) from this paper. In particular (see \thmref{smodp}~$(b)$) we show that the quantity $\Pi(x) = \pi_F(x, \det; \mrm{sf}) $ (cf. section~\ref{siegquanta}) defined by
\begin{equation}
\Pi(x):= \{ d \le x |d \text{ square-free, }    a_F(T) \nl \text{ for some } T \in \Lambda^+_n \text{ such that } \det(T) =d \} 
\end{equation}
satisfies for all $\ell$ sufficiently large (see \rmkref{lexpl}) the lower bound
\begin{equation} \label{Pi}
\# \Pi(x) \gg x/(\log x)^{\beta(F)},  \qq (n \text{ odd })
\end{equation}
for some constant $0 < \beta(F) \le 3/4$ and implied constant depending on $F$ and for all $\mk l$ lying over $\ell$. This was in fact our main motivation for writing this paper, and can be viewed as a $\ml$ version of the recent result \cite[Theorem~1]{bo-da} on non-zero `fundamental' Fourier coefficients of Siegel modular forms. An inspection of \cite[Theorem~1]{bo-da} shows that when $n$ is odd and $F$ as above, \eqref{Pi} is an improvement on the lower bound (viz. $x^{1-\epsilon}$) on $\Pi_\infty(x)$, which is defined exactly as $\Pi$ above, but we now count $T$ with $a_F(T) \ne 0$. It is easily noted that in \cite[Theorem~1]{bo-da}, the lower bounds of the form $\Pi_\infty(x) \gg x^{1-\epsilon}$ can be slightly improved to $\Pi_\infty(x) \gg x \exp(-c \log x/ \log \log x)$ for some $c>0$ (the improvement coming from using the bound $\sigma_0(n) \le \exp(c \log n/ \log \log 2n)$ instead of $\sigma_0(n) \ll n^{\epsilon}$). However \eqref{Pi} is still better for $F$ with algebraic Fourier coefficients. 

Finally let us mention that to obtain \eqref{Pi}, we actually use its archemedian analogue (only the existence of a nonvanishing fundamental Fourier coefficient) from \cite{bo-da} as an input. So we obtain no new proof of it, even though such a thing is desirable (a preliminary inspection shows that even then $\ell$ has to be large), and seems hard. Moreover since our results, say for $\Pi(x)$, hold only for large enough $\ell$, merely having $\Pi(x)>0$ is a tautology as we can fix $T$such that $a_F(T) \ne 0$ and remove the finitely many $\ell$ such that $\ell \mid a_F(T)$. We therefore must aim for statements $\Pi(x) \to \infty$ as $x \to \infty$. The same remark applies to the nonvanishing $\bl$ of primitive Fourier coefficients as well.
Our method should work for Hermitian modular forms of any degree, however.

Finally for the reader's convenience, let us mention that the only place where we use the results from \cite{j, js, jsh} are in \thmref{smodp} and \propref{strace}. Alongwith this in the proof of \thmref{smodp}, we also use results from section~\ref{oldie}, precisely \thmref{emf}.

\subsection*{Acknowledgements}
{\small S.D. was supported by a Humboldt Fellowship from the  Alexander von Humboldt Foundation at Universit\"{a}t Mannheim during the preparation of the paper, and thanks both for the generous support and for providing excellent working conditions. The authors acknowledge the use of the LMFDB databse for some numerical computations. S.D also thanks IISc. Bangalore, DST (India) and UGC-CAS for financial support. During the preparation of this work S.D. was supported by a MATRICS grant MTR/2017/000496 from DST-SERB, India. We thank M. Raum for his comments on Siegel modular forms $\bmod p$.}

\section{Setting and notation} \label{prelim}
Following standard notation, let $\mN$ and $\mnk$ denote the space of elliptic (respectively Siegel) modular forms of weight $k$ and level $\Gamma_1(N)$ (respectively $\Gamma^n_1(N)$). We denote their Fourier expansions as follows ($\mf H$ and $\mf H_n$ being the respective upper-half spaces):
\begin{align}
f(\tau) &= \sum_{n \ge 0} a(f,n) e( n \tau) \qq (\tau \in \mf H, f \in \mN) \\
F(Z) &= \sum_{T \in \Lambda_n} a_F(T) e(TZ), \qq (Z \in \mf H_n, F \in \mnk)
\end{align}
where $e(z) = \exp(2 \pi i z)$ for $z \in \mf C$, $e(TZ) = e(\mrm{trace}(TZ))$, $\Lambda_n$ denotes the set of half-integral positive semi-definite symmetric matrices over $\mf Z$. We also put $\Lambda_n^+$ to be the positive-definite elements of $\Lambda_n$. The corresponding spaces of cusp forms are denoted by $\sN$ and $\snk$. To avoid any confusion, let us mention that for this paper 
\[  \Gamma^n_1(N):= \{ \gamma = \begin{psm}
A & B \\ C& D \end{psm}  \in \Gamma_n | \det(A) \equiv \det(D) \equiv 1 \bmod N, \, C \equiv 0 \bmod N
\}. \]

Further, if $\mc O_K$ is the ring of integers of a number field $K$, we put
\begin{align}
\mnk(\mc O_K) := \{ F \in \mnk | a_F(T) \in \mc O_K  \text{ for all } T\} ,\\
\widehat{M}_k(N,\ok):= \{ F \in \mN | a(f,n) \in \mc O_K \text{ for all } n \ge 1, a(f,0) \in K \}. \label{mhat}
\end{align}
We denote by $\sN(\mc O_K), \snk(\mc O_K), \widehat{S}_k(N,\ok) $ to be the respective spaces of cusp forms. Note that $\widehat{S}_k(N,\ok) = \sN(\mc O_K)$.

For two positive functions on $\mf R$, we write $f(x) \asymp g(x)$ if there exist two positive constants $c_1,c_2 $ such that $c_1 g(x) \le f(x) \le c_2 g(x)$ for all $x \ge 1$.

We now recall some notions about Fourier-Jacobi expansions of Siegel modular forms.
We first recall that the content $c(T)$ for any matrix $T=(t_{i,j}) \in \Lambda_n$ is defined as $\gcd$ of all
the $t_{ii}$ and all the $2t_{i,j}$ with $i\not=j$.
In particular, $T$ is called primitive, if $c(T)=1$.

For  a fixed  $S\in \Lambda^+_{n-1}$  we consider a Jacobi form 
$\varphi({\mk Z})=\phi_S(\tau, z)e(SZ)$ (where $\mk Z = \begin{psm}  \tau & z \\ z & Z  \end{psm} \in \mf H_n$) of index $S$ on the group $\Gamma_1(N) \ltimes \mf Z^{n-1}$.
Its theta expansion has the form 
\begin{equation}
\phi_S(\tau,z)=\sum_{\mu_0} h_{\mu_0}(\tau)\Theta_S[\mu_0](\tau,z) \label{jacobi1} 
\end{equation}
where $\mu_0$ runs over ${\mf Z}^{n-1}/ 2S {\mf Z}^{n-1}.$ We write the Fourier expansions of $\phi_S$ and $h_{\mu_0}$ as
\begin{equation}
\phi_S= \sum_{r,\mu} b(r,\mu) e(r\tau +\mu^t\cdot  z ), \q h_{\mu_0}(\tau)=\sum_r b(r,\mu_0) e (r-S^{-1}[\mu_0/2])\cdot \tau)
\label{jacobi2}
\end{equation}
with $r\in {\mf N}_0$, $\mu\in {\mf Z}^{(n-1,1)}$ and for all $L\in {\mf Z}^{(n-1,1)}$, note the following invariance property
\begin{equation} 
b(r,\mu)= b(r+L^t\cdot \mu +S[L]  ,\mu + 2S\cdot L). \label{jacobi3}
\end{equation}
We would be mainly interested in the Fourier-Jacobi expansion of $F \in \mnk$ of type $(1,n-1)$:
\begin{equation} \label{fj}
F(\mk Z)= \sum_{S \in \Lambda_{n-1}} \phi_{S}(\tau, z) e(SZ) \q \q (\mk Z = \begin{psm}  \tau & z \\ z & Z  \end{psm});
\end{equation}
then the $\phi_S$ are Jacobi forms in the above sense.

\section{Elliptic modular forms} \label{ellsec}
Let $f \in \mNok$. Now by the classical theory of newforms, there exist $\alpha_{i,\delta} \in \mf C$ such that $f(\tau)$ can be written uniquely in the form (see \cite{miyake} for cusp forms, and \cite{weis} for Eisenstein series) with apriori complex numbers $\alpha_{i,\delta}$ such that (the $f_i$ traverse through newforms of level dividing $N$, including Eisenstein series)
\begin{equation}\label{eq:decomp}
f(\tau)=\sum_{i=1}^{s} \sum_{\delta |N} \alpha_{i,\delta}f_{i}(\delta \tau).
\end{equation}
We note here that our normalisation for the Eisenstein newforms is that it's Fourier coefficient at $n=1$ equals $1$; i.e., we take, for two primitive Dirichlet characters $\chi_1,\chi_2 \bmod N$ with $\chi_1, \chi_2= \chi, \chi(-1)=(-1)^k$ and consider the newforms
\begin{align}
E_{\chi_1,\chi_2}(\tau):= \delta_{1,\chi_1} L(1-k,\chi_1) + \sum_{n \ge 1} \left( \sum_{d \mid n}\chi_1(n/d) \chi_2(d) d^{k-1} \right).
\end{align}
Here $\delta_{1,\chi_1}$ is $1$ or $0$ according as $\chi_1$ is principal or not. Moreover $L(s,\chi_1)$ is the Dirichlet $L$-function attached to $\chi_1$. In the above notation, $E_{\chi_1,\chi_2} \in \widehat{M}_k(N,\ok)$.
Since we would be mostly interested in counting the number of non-zero Fourier coefficients $a(f,n)$ for $1 \le n \le x$, with a large parameter $x$, this choice of normalisation would be sufficient for our purpose.

We would prove that the $\alpha_{i,\delta}$ are all $\mk l$-integral for suitable primes $\mk l$, see \lemref{lint}.
We recall here the Sturm's bound for $\mN$: if $G,H \in \mN(\mc O_K)$ and satisfy $a(G,n) \equiv a(H,n) \ml$ for some prime $\mk l \subset \mc O_K$, for all 
\begin{equation} \label{sturm}
n \le \mc S^1(k,N) := \frac{k}{12} [\sltwo \colon \Gamma_1(N)],
\end{equation}
then $G \equiv H \ml$. Later we would write down a similar bound for Siegel modular forms. We would also require the archemedian version of the Sturm's bound, and note that the bound in \eqref{sturm} also works in this case.

\textsl{We assume without loss of generality that $K$ contains the eigenvalues $b_i(n)$ (for all $i$ and $n \geq 1$) and is Galois over $\mf Q$}. See \rmkref{klarge}.

\subsection{(Non)-Congruences, analytic way}  \label{noncongana}
For our further requirements, we need to separate Hecke eigenvalues of two newforms $\ml$ in an efficient way (for suitable $\mk l$). For this we first discuss an analytic way relying essentially on Deligne's bound. In the next subsection, we sketch a possible simple algebraic way to do this in certain special situations.

Let us now focus on \eqref{eq:decomp}.
For all primes $p$ and $1 \le i \le s$, one has $T_{p}f_{i}=b_{i}(p)f_{i}$ ($f_i,b_i$ as in \eqref{newfc}). By "multiplicity-one" (see \cite{weis} for Eisenstein series), if $i\neq j$, we can find infinitely many primes $p$ \textsl{coprime to $N$} such that $b_{i}(p)\neq b_{j}(p)$. We would like to have a more precise $\ml$ version of this, with effective bounds on $p$. 

Let us consider a set of primes $q_{i,j}$ all coprime to $N$ such that for all $i \ne j$, one has $b_i(q_{i,j})  \neq b_j(q_{i,j})$ with $b_i$ as above. Let us define the non-zero quantity
\begin{equation}
\mc L(\{ q_{i,j} \};f) := \prod^{\prime}_{(i, j) \in S_f \times S_f} \left( b_i(q_{i,j}) - b_j(q_{i,j}) \right) 
\end{equation}
where $S_f$ is the set of indices $i$ such that for which $f_i |B_\delta$ (for some $\delta \mid N$) appears in $f$ (cf. \eqref{eq:decomp}), i.e., 
\begin{equation} \label{sf}
S_f = \{ 1 \le i \le s |  \alpha_{i,\delta} \neq 0 \text{ for some } \delta \mid N  \}.
\end{equation}
Moreover $\overset{\prime}\prod$ denotes the product over indices $i \neq j$. 

We call any such set of primes $\{ q_{i,j} \}$ ($i,j \in S_f$)  as {\bf admissible} for $f$. By multiplicity-one, there are infinitely many admissible sets for any $f$.

\begin{lem} \label{lint}
Let $\{ q_{i,j} \}$ be any set of admissible primes for $f$. With the notation and setting as above, the $\alpha_{i,\delta}$ appearing in \eqref{eq:decomp} are all $\mk l$-integral for any prime $\mk l \subset \ok$ such that $\mk l \nmid \mc L(\{ q_{i,j} \};f) $.
\end{lem}
We postpone the proof of the lemma until that of the next proposition. In fact the proof of the lemma follows the lines of that of \propref{mainprop} given below, and this explains our choice.

\subsubsection{Ensuring non-congruences by using analytic estimates}  \label{anasec}
Recall our convention on $K$ from the previous section.
Now suppose that for an \textsl{arbitrary} prime $q$ such that $b_1(q) \neq b_2(q)$ there is a congruence 
\begin{equation} \label{congij}
b_1(q) \equiv b_2(q) \bmod \mk l .
\end{equation} 
If we take norms on both sides of \eqref{congij}, we get the divisibility relation in $\mf Z$:
\begin{equation}\label{ldiv}
N(\mk l ) \, | \, N \left( (b_1(q) - b_2(q) ) \right).
\end{equation} 
We further let $\mk d = [K \colon \mf Q]$ and $\mk h$ to be the inertia degree of $\mk l$ over $\ell$.

The point to note here is that by virtue of Deligne's bound and Shimura's result about the existence of a basis of $\sN$ with (rational) integral Fourier coefficients which implies that for any $\sigma \in Aut(\mf C)$, $g^\sigma := \sum_n \sigma(a_g(n))q^n \in \sN$ whenever $g$ is; we conclude from \eqref{ldiv} that if 
$\ell > (4q^{\frac{k-1}{2} })^{\mk d /\mk h } $, then the congruences \eqref{congij} cannot hold at $q$. In particular if 
\begin{align}
\ell > (4P^{\frac{k-1}{2} })^{\mk d / \mk h} 
\end{align}
where $P= \prod_{i \ne j} m_{i,j}$ with $m_{i,j}$ as in \eqref{ldef}, then $\mk l \nmid \mc L$. In the next section we show how to bound the $m_{i,j}$, in a slightly more general situation.


\subsubsection{Effective bounds for the $m_{i,j}$ and ensuring $\mk l \nmid \mc L(\{ m_{i,j} \};f)$} \label{mij}
For future applications, we need to choose primes $\{ q_{i,j}\}$ separating the newforms $\{f_j  \}$ such that $(q,2NQ)=1$, where $Q$ is an arbitrary given integer. We would first bound the elements of the "smallest" such admissible set of primes (each coprime to a fixed integer) in terms of $k,N$, and then provide a constant $\mc C= \mc C(k,N)$ such that for all $\ell > \mc C$ and all $\mk l \mid \ell$, one has $\mk l \nmid L(\{ q_{i,j}\} ;f)$.

To this end we consider as before, for any $ f_i \in \{ f_1, \ldots, f_s\}$ and for any given $M \ge 1$, the modified modular form
\begin{equation} \label{fm}
 f_i^{(M)} (\tau) := \sum_{(n,M)=1} b_i(n)q^n  := \sum_{n \ge 1} \mf b_i(n)q^n \in M_k(\Gamma_1(\widetilde M),
\end{equation}
where $\widetilde M = N^2$ if $M=N$, and $NM^2$ otherwise.
Observe that all of the $ f_i^{(M)}$ are non-zero and moreover that $f^{(M)}_i \neq f^{(M)}_j $ whenever $i \ne j$. Now there are two cases, and our treatment would be slightly different in each case.

\subparagraph{(A)~\textsl{Not necessarily distinct primes}} \label{qij}
We want to choose primes $q_{i,j}$ such that for all $i \ne j$,
\begin{equation}
b_i(q_{i,j}) \not \equiv b_j(q_{i,j}) 
\end{equation}
and $(q_{i,j},M)=1$. We simply choose the $\mf m_{i,j}$ to be the smallest prime (necessarily less or equal to the Sturm's bound $\mc S^1(k,\widetilde M^2)$) such that $\mf b_i(\mf m_{i,j}) \neq \mf b_j(\mf m_{i,j})$ and then choose $\mk l$ such that $\mk l \nmid \mc L^{(M)}(\{ \mf m_{i,j} \} ;f)$ where (recall $S_f$ from \eqref{sf})
\begin{equation} \label{l0def}
\mc L^{(M)}(\{ \mf m_{i,j} \} ;f) := \prod^\prime_{i, j \in S_f \times S_f} \left( b_i(\mf m_{i,j}) - b_j(\mf m_{i,j}) \right)
\end{equation}
Note that in particular $\max_j \{ \mf m_{i,j} \} \le \mc S(k,\widetilde M^2)$ and from our definition \eqref{fm}, necessarily $(\mf m_{i,j},M)=1$. 

Now by the arguments in section~\ref{mij} see that the following $\ell$ (and any $\mk l$ lying over $\ell$) work:
\begin{equation} \label{l0}
\ell >  (4 {S(k, \widetilde M^2)}^{\frac{s(s-1)(k-1)}{4}  })^{\frac{\mk d}{\mk f}}.
\end{equation}
In particular when $M=N$, we have $\mf m_{i,j}=m_{i,j}$ (cf. \eqref{ldef}) and a thus an effective upper bound for $\mc L(\{ m_{i,j} \};f)$.

\subparagraph{(B)~\textsl{Distinct primes}} \label{pij}
If we insist that the primes $q_{i,j}$ requested as above are also pairwise distinct, we can proceed similarly as in (a) with some modifications.

We start by picking the prime $p_{1,2}:=\mf m_{1,2}$ coprime to $M$ from an admissible set for $f$ as in (a) above. We then reiterate this procedure as follows. Next we consider the forms $f_1^{(Mp_{1,2})}$ and $f_3^{(Mp_{1,2})}$ (both are non-zero modular forms) and find a prime $p_{1,3} \le \mc S(k, NM^2 p_{1,2}^2)$ such that $b_1(q) \not \equiv b_3(q)  \bmod \ml$ and $(p_{1,3}, M p_{1,2})=1$. We carry on this procedure to get primes $p_{1,j}$ ($2 \le j \le s$) satisfying $(p_{1,j}, Mp_{1,2}\cdots p_{1,j-1})=1$. We finally do this for the indices other than $1$ and finally find that
$\max_{i,j} \{ p_{i,j} \} \ll \mk S(k,N)$, where
\begin{equation} \label{psize}
\mk S(k,N,M) = \mc S(k, N M^2 \prod_{i<j} p_{i,j}^2 ), \q p_{i,j} \le \mc S(k, NM^2 \prod_{i<t <j}p_{i,t}^2 ).
\end{equation}
We thus consider $\mk l \nmid \mc L_{\mrm{sf}}(M)$ where
\begin{equation} \label{lmdef}
\mc L^{(M)}_{\mrm{sf}}(\{ p_{i,j}\}; f) := \prod^{\prime}_{i, j \in S_f \times S_f} \left( b_i(p_{i,j}) - b_j(p_{i,j}) \right).
\end{equation}
If $M=1$, we omit it from the notation.

By using Deligne's bound we again deduce that if
\begin{equation} \label{lineq}
\ell>  (4 {\mk S(k,N,M)}^{\frac{s(s-1)(k-1)}{4}  })^{\mk d / \mk h},
\end{equation}
none of the congruences $b_i(p_{i,j}) \equiv b_j(p_{i,j}) \ml$ in particular hold for $i \ne j$ where $(p_{i,j},M)=1$ and $p_{i,j}$ are pairwise distinct. 

We can now state the main workhorse of this paper. If $f$ is not a constant $\ml$ for suitable $\mk l$, we show that by considering the Hecke module generated by $f$, one can extract a newform $\mk g$ ocurring in $f$, such that the Fourier coefficients of $\mk g$ are integral linear combinations of those of $f$. This would allow us to reduce our nonvanishing questions to those about newforms. Also from the point of view of this paper, as we discussed before, singular forms play no role and thus should be avoided. 

\begin{prop} \label{mainprop}
Let $f$ be as in \eqref{eq:decomp} and let $Q \ge 1$ be arbitrary. Let $\{q_{i,j}\}$ be a set of admissible primes for $f$.
Then there exists a  newform $\mk g$ of level dividing $N$ such that one has the relation
\begin{equation} \label{maineq}
a(\mk g,\mk n) \equiv \sum_{t |N} \beta_{t}a(f,\gamma_{t} \Delta \mk n) \ml
\end{equation}
for all $\mk n \ge 1$ such that $(\mk n, QN)=1$, for some $\Delta | N$; for some $\beta_t$ which are $\mk l$-integral; and where

(a) $\gamma_t \in \mf Q$ are all square-free with $(\gamma_t,QN)=1$ provided $\mk l \nmid \cdot \mc L^{(QN)}_{\mrm{sf}}(\{ q_{i,j}\};f)$ \textsl{and} $f$ is not constant $\ml$;

(b) $\gamma_t \in \mf Q$ with $(\gamma_t,QN)=1$ provided $\ell \nmid  \cdot \mc L^{(QN)}(\{ q_{i,j}\};f)$ \textsl{and} $f$ is not constant $\ml$.
\end{prop}
Note that the set $\{  p_{i,j}\}$ constructed in section~\ref{mij}~(B) above satisfy the properties requested for the admissible set in~$(a)$. Similarly the set $\{ \mf m_{i,j}\}$ is an example for the admissible sets in~$(b)$.

\begin{proof}
For ease of notation, in $(a)$ let us assume without loss of generality that the admissible set is $\{ p_{i,j} \}$. Also we put $\mc L_1:= \mc L^{(QN)}_{\mrm{sf}}(\{ p_{i,j}\};f)$ and $\mc L_2:=\mc L^{(QN)}(\{ q_{i,j}\};f)$.

By \lemref{lint} we know  that all the $\alpha_{i,\delta}$ are $\mk l$-integral. By our assumption on $f$, not all of them can be $\equiv 0 \bmod \mk l$. We may, after renumbering the indices, assume $\alpha_{1,\delta}\not \nl$ for some $\delta|N$. Let $q:=p_{1,2}$ or $q:= q_{1,2}$ according as we are in case~$(a)$ or ~$(b)$ be the prime chosen as before (see the discussion preceeding this theorem) for which $b_{1}(q)\not \equiv b_{2}(q) \bmod \mk l$. Note that $p_{1,2} \nmid QN, q_{1,2} \nmid N$. Then consider the form $g_{1}(\tau)=\sum_{n=1}^{\infty}a_{1}(n)q^{n} :=T(q) f(\tau)-b_{2}(q)f(\tau)$ so that
\begin{align*}
g_{1}(\tau)=\sum_{i=1}^{s}(b_{i}(q)-b_{2}(q))\sum_{\delta|N} \alpha_{i,\delta}f_{i}(\delta \tau).
\end{align*}
The modular forms $f_{2}(\delta \tau)$ for any $\delta \mid N$, do not appear in the decomposition of $g_{1}(\tau)$ but $f_{1}(\delta \tau)$ does for some $\delta |N$. Proceeding inductively in this way, we can remove all the non-zero newform components $f_{i}(\delta \tau)$ for all $i=2,...,s$, to obtain a modular form $F$ ($=g_{|S_f|-1}$ in the above inductive procedure) in $\mN$ such that on the one hand, we have
\begin{equation} \label{gs-1}
F(\tau)=\sum_{n=1}^{\infty}A(n)q^{n}:=\prod_{2 \le j \le s}  (b_{1}(p_{1,j})-b_{j}(p_{1,j}))    \sum_{\delta|N}\alpha_{1,\delta}f_{1}(\delta \tau).
\end{equation}

By the construction of admissible sets, the product in \eqref{gs-1} is $\nl$, provided $\mk l \nmid \mc L_1$ or $\mk l \nmid \mc L_2$ according as we are in case $(a)$ or $(b)$. Therefore rescaling $F$, and calling the resulting function again as $F$, we on the other hand note that the inductive procedure gives us finitely many algebraic numbers $\beta_{t}$ (polynomials in the $p_{i,j}$'s or the $q_{i,j}$'s and Dirichlet characters) and positive rational numbers $\gamma_{t}$ (which  quotients of the $p_{i,j}$'s or the $q_{i,j}$'s) such that for every $n$
\begin{equation} \label{f}
A(n) =\sum_{\delta|N} \alpha_{1,\delta} b_{1}(n/\delta) \equiv \sum_{t} \beta_{t}a(f,\gamma_{t}n) \ml.
\end{equation}
Let $\delta_{1}$ be the smallest divisor of $N$ such that $\alpha_{1,\delta_{1}}\nl$ in \eqref{f}. Then choosing $n= \mk n \delta_1$ in \eqref{f} such that $(\mk n, QN)=1$, we get
\begin{equation} \label{neweq}
\alpha_{1,\delta_1}b_1(\mk n) \equiv \sum_t \beta_{t}a(f,\gamma_{t} \delta_1 \mk n) \ml
\end{equation}
as desired. The square-freeness of $\gamma_t$ in part~(a) follows from the pairwise distinctness of the $p_{i,j}$ (cf. section~\ref{mij}~(B) preceeding this theorem, and from the formula for the action of Hecke operators at primes). This completes the proof.
\end{proof}

\begin{rmk}
It is obvious from the above proof that we do not really need all the primes $q_{i,j}$ from an admissible set for $f$ to make the proof work. Only primes $q_{i_0,j}$ with $j \neq i_0$ (eg. we assumed $i_0=1$ in the above proof) for some fixed $i_0$ such that $\alpha_{i_0,\cdot} \nl$ are required. Accordingly one could have modified the definition of $\mc L(\cdots)$. We do not do this to avoid additional notational burden. However in practice, this is what should be done.
\end{rmk}

\begin{proof}[Proof of \lemref{lint}]
The proof is essentially contained in the foregoing proof (of \thmref{emf}) itself, if we take $\ell=\infty$ (see \cite{ad}). We refer to that proof for all the subsequent arguments. Namely we start with $\alpha_{i,\delta} \in \mf C$ in \eqref{eq:decomp}. If $i$ is an index for which $\alpha_{i,\delta} \ne 0$ (for some $\delta$) we carry out the procedure of removing the newforms one by one, we would arrive at \eqref{neweq} corresponding to $i$. Here we choose $\delta_1$ to be the minimum such that $\alpha_{i,\delta_1} \ne 0$ (not $\bmod \mk l$!). Then choosing $n=\delta_1$ in \eqref{neweq} shows that $\alpha_{i,\delta_1} \cdot \prod_{2 \le j \le s, j \neq i}  (b_{1}(q_{i,j})-b_{j}(q_{i,j}))  $ is an algebraic integer (in fact in $\ok$). Note that the product appears since we had normalised $F$ by this factor in the proof of \thmref{mainprop}. Thus by the definition of an admissible set, if $\mk l \nmid \mc L(\{ q_{i,j} \};f) $ (in particular if $\ell$ is large enough, cf. section~\ref{anasec}) $\alpha_{i,\delta} $ will be $\mk l$-integral. Then inductively, we can ensure that all the $\alpha_{i,\delta} $ are $\mk l$-integral. Since $i$ was arbitrary, we get the lemma.
\end{proof}

\begin{rmk}\label{lint-litchen}
We point out another way of proving \lemref{lint} by using some arguments from \cite[Proof of Lemma~2.1]{lit}. Working with the Sturm's bound for $\mN$ (see \cite{murty}) as in our proof, this boils down to ensuring that $\nu_{\ell} (\det ( b_{i,t}(n_j) ))=0$; where $f_i(t \tau) = \sum_n b_{i,t}(n) q^n$ with $\{ f_i \}$ is the set of all newforms in $\mN$ of level dividing $N$, and for each such $i$, $t=t(i)$ runs over divisors of $N/\mrm{level }(f_i)$. Moreover $\{n_j\}$ are certain indices less than the Sturm's bound for $\mN$. It is however not clear how one can ensure that $\nu_\ell(\cdots)=0$ without using `analytic' inputs as in our paper, it does not help that we do not have much control on the set $\{n_j \}$ from \cite{lit}. If analytic inputs are used, then the ensuing bound on $\ell$ would be similar to ours.
\end{rmk}

\subsection{(Non)-Congruences, algebraic way} \label{congalg}
There is of course an algebraic way to ensure non-congruences. Even though we would not really use this in the sequel (since we face some trouble, see below), except for some examples, we include this for completeness and possible further interest.

For each $1 \le i \le s$, from \cite{hida1} let $c(f_i)$ denote the number whose square (which is in $ \mf Z$) determines the congruence primes for $f_i$. {\bf We assume that the Galois-orbits of each $f_i$ is singleton} (i.e. $f_i$ has Fourier coefficients in $\mf Z$). We note that $c(f_i)$ is closely related to the adjoint $L$-value $L(1, Ad(f_i))$, and there is the principle: ``the prime factors of the denominators of the adjoint $L$-value give the congruence primes of $f_i$". This is proved in the seminal papers \cite{hida1, hida2, rib1}. Let us put
\begin{equation}
\mc C(k,N) = \mrm{lcm}\{  c^2(f_1),  \ldots, c^2(f_s) \}
\end{equation}

Then from the aforementioned references, for any prime 
\begin{equation} \label{lcond}
\ell \nmid 6N\mc C(k,N)  \text{  such that  } \ell \ge k,
\end{equation}
one can ensure that for all $i \neq j$, and all $\mk L$ lying over such $\ell$ in $\overline{\mf Z}$,
\begin{equation}
f_i \not \equiv f_j \bmod \mk L.
\end{equation}
This will imply the existence of a prime $q$ with $b_i(q) \not \equiv b_i(q) \bmod \mk L$, and {\bf let us assume} $(q,N)=1$.
We do not know how to remove the assumption on Galois-orbits algebraically. However when $k=2$, there are results by Ventosa (\cite[Lem.~2.10.1]{ventosa}) which say that if a newform $g$ is congruent to one of it's Galois-conjugates, then the congruence prime must divide the discriminant of the polynomial $Q_{g,p}(X) = \prod_{\sigma} (X - \sigma(a(g,p)))$ for any prime $p \nmid N\ell$. So in this case we would be able to produce admissible sets based on the remarks above, and therefore our method would work for all $\ell \ge k$ and $\ell \nmid \text{ (some fixed integer)}$.

\subsubsection{An application of \propref{mainprop} to oldforms $\ml$} \label{oldie}
\propref{mainprop} has many consequences, which we now discuss in the next couple of results. The first one is a statement about "old-form theory $ \ml$". Let $f \in \mN(\ok)$ and $\mk l \mid \ell$. Consider any admissible set $\{  q_{i,j} \}$ whose elements are coprime to $M=QN$. The admissible set $\{ \mf m_{i,j} \}$ from section~\ref{mij}~(A) is an example.

\begin{thm} \label{emf}
Let $f$ be as above and suppose that $f$ is not singular $\ml$. Let $Q \ge 1$ be given. Assume that $\mk l \nmid \mc L^{(QN)}( \{ q_{i,j}\};f)$ for any admissible set $ \{  q_{i,j} \}$ such that $(q_{i,j},QN)=1$. Also assume $\ell>2$. Then the following hold.

(a) If $(Q,N)=1$, there exists infinitely many integers $n \ge 1$ such that $a_f(n) \not \equiv 0 \bmod \mk l$ and $(n,Q)=1$.

(b) If $(Q,N)>1$, and we request that the $f$ is not congruent $\ml$ to an $\mk l$-integral linear combination of modular forms $g \mid B_{M_g}$ for some $g \in \widehat{M}_k(N,\ok)$ ($\widehat{M}$ as in \eqref{mhat}) of level $M_g$ dividing $(Q,N)$ and $M_g >1$; the same conclusion as in (a) above holds. (Here $g \mid B_{M_g}(\tau) = g(M_g \tau)$.)
\end{thm}
Note that $\widehat{S}_k(N,\ok) = S_k(\Gamma_1(N))$, so in the space of cusp forms \thmref{emf} is a generalisation $\ml$ of the classical old-form theory over $\mf C$ for all but finitely many $\ell$.

\begin{proof}
For the proof we have to go back to \eqref{maineq} in \propref{mainprop}. It is easy to see that since $\ell$ is odd there are infinitely many primes $p$ with $(p, QN)=1$, such that $b_1(p) \nl$. This follows from the fact the set of primes such that $b_1(p) \equiv 0 \bmod \mk l$ is `Frobenian' of density less than $1$ (see \cite[cf. proof of \texttt{Th\'{e}or\`{e}me~4.7}~$\mathtt{(ii_1)}$ on p.~13]{serre}). 

Therefore if $(Q,N)=1$, in \eqref{maineq} we see that there exists $t$ such that $n=\gamma_t \delta \mk n$ is coprime to $Q$ and $a(f,n) \nl$. This proves part~$(a)$.

If $(Q,N)>1$ and the condition in the statement of part~$(b)$ is satisfied, this will imply that in \eqref{eq:decomp} $\alpha_{i,1} \nl$ for at least one $1 \le i \le s$, say $i=i(0)$. Then we would remove all the newforms except $f_{i(0)}$ and arrive at \eqref{neweq}. However now, we can choose $\delta_1=1$ in \eqref{neweq} and hence we get part~$(b)$.
\end{proof}

\begin{rmk} \label{lrem}
It follows from the above proof that if the level $N$ is square-free, then in \thmref{emf} we can ensure that $n$ is also square-free for $\mk l$ as in the theorem.
\end{rmk}

Of course \thmref{emf} can be rephrased as a $\mod \mk l$ version of the familiar result in oldform theory for elliptic modular forms over $\mf C$.
\begin{cor} \label{oldmod}
Let $f$ and $\ell$ be as in \thmref{emf}. Suppse that for some $Q \ge 1$, $a_f(n) \equiv 0 \bmod \mk l$ for all $n$ with $(n,Q)=1$. Then $f \equiv 0 \ml$ if $(Q,N)=1$, otherwise $f$ is ``old'' $ \ml$; i.e., there exist primes $\ell_1,\ldots,\ell_m$ such that $\ell_j | (Q,N)$ and $f \equiv \sum_j \alpha_j \, h_j  | B_{\ell_j} \ml $ for some $h_j \in \widehat{M}_k(\Gamma_1(N_j))$, where $\ell_j N_j |N$, and $\alpha_j$ being $\mk l$ integral. 
\end{cor} 

\subsection{Quantitative results: elliptic modular forms} \label{ellquan}
In \cite{js} (resp. \cite{jsh}) an asymptotic formula (resp. a lower bound) for the number of $\bl$-non-zero Fourier coefficients of modular forms of integral weights (resp. half-integral) of level $N$ was obtained.
Following the notation in \cite{js} let us define for $f \in \mN(\ok)$ the counting function $\pi(f,x)$ by
\begin{align*}
\pi(f,x; \mk l) := \{ n \le x | a(f,n) \nl   \}; \q  &\pi_Q(f,x; \mk l) := \{ n \le x, (n,Q)=1 | a(f,n) \nl   \} \\
\pi_{\mrm{sf}}(f,x; \mk l) := \{ n \le x & |n \text{ square-free}, \ a(f,n) \nl   \},
\end{align*}
and $\pi_{\mrm{sf}, Q}(f,x; \mk l)$ defined analogously.

Here we want to show that quite elementarily using \propref{mainprop} one can at least obtain upper and lower bounds of same order of magnitude for the quantities $\pi(f,x; \mk l)$ and $\pi_{\mrm{sf}}(f,x; \mk l)$, however for $f \in \mN$ of fixed weight and level. Similar considerations already seem to appear in \cite{baono}. These results in turn imply elementarily analogous results for Siegel modular forms.

Let $\{ q_{i,j}  \}$ be an admissible set for $f$ and recall the admissible set $\{ p_{i,j}  \}$ constructed in section~\ref{mij}~(B). We define the integers $P:= \prod_{i,j} q_{i,j}$ and $U = \prod_{i \ne j} p_{i,j}$, ($i,j \in S_f$).

\begin{prop} \label{ellasy}
Let the setting be as above. Then
\begin{align}
| \pi(f,x; \mk l) |  \asymp x/  (\log x)^{\alpha(f)}  \q &( N \ge 1, \mk l \nmid \mc L(  \{ q_{i,j}  \} ;f)); \label{asymp1}  \\
| \pi_{\mrm{sf}}(f,x) | \asymp x/(\log x)^{\alpha(f)} , \q & (N \text{ square-free, } \mk l \nmid \mc L_{\mrm{sf}}(\{ p_{i,j}\}; f)). \label{lubd}
\end{align}
The implied constants depend only on $k,N,P,U$, and $0<\alpha(f) \le 3/4$.
\end{prop}

\begin{rmk} \label{klarge}
Note that we assumed $K$ to be a bit large -- containing the field $L$ generated by all the eigenvalues of newforms of level dividing $N$. However this is just for notational convenience; one could have passed to the normal closure, say $\widetilde K$ of $KL$ and noted that $\pi(f,x; \mk B) = \pi(f,x; \mk l)$ for any $\mk B \in \widetilde K$ lying over $\mk l$.
\end{rmk}

\begin{proof}
We start from \eqref{maineq} in \propref{mainprop} in the special case when $Q=N=1$.
For each $\mk n$ such that $(\mk n,N)=1$ and $a(\mk g,\mk n) \nl$, from \eqref{maineq} there exists a $t \mid N$ such that $\beta_t a(f, \mk n \gamma_t \Delta) \nl$. Choose the smallest such $\mk n$, call it $t(\mk n)$. Then the map 
\begin{equation} \label{inj1}
\pi_{PN \ell }(\mk g,x) \longrightarrow \pi(f, PN  x), \q \mk n \mapsto \gamma_{t(\mk n)} \Delta \mk n,
\end{equation}
where $P= \prod_{i,j} q_{i,j}$ as recalled above, is clearly injective since $\mk n$ is away from $PN$. Note that $\gamma_{t(\mk n)}$, in this case, is a polynomial in the $p_{i,j}$s.
Therefore for all $x \ge 1$,
\begin{equation}
\#\pi_{PN \ell }(\mk g,x) \le \#\pi(f, PN  x);
\end{equation}
and the lower bound in \eqref{asymp1} follows from \eqref{inj1} and from results of Serre \cite[\S~(4.6)]{serre} for newforms. Clearly Serre's asymptotic formulae for $\pi_{PN \ell }(\mk g,x)$ also hold if we omit finitely many primes from the ensemble that he starts with (loc. cit.). More precisely from \cite{serre}, via the Galois representation attached with $\mk g$, the primes $p \nmid N\ell$ are "Frobenian" and has analytic density $\alpha(\mk g)$ with $0 < \alpha(\mk g) \le 3/4$ (cf. \cite[(6.3)]{serre}); and Serre shows that from this one can, using analytic techniques, give an asymptotic formula for $\pi_{PN \ell }(\mk g,x) $. Since the asymptotic formula depends only on the density of primes concerned, our claim holds true. We take $\alpha(f):= \alpha(\mk g)$.

For the upper bounds, we look instead at \eqref{eq:decomp} and apply the same argument just presented. Here we consider the map $\mk n \mapsto \left(\, \{ i(\mk n), \delta(\mk n) \}, \mk n /\delta(\mk n) \right)$ where $i(\mk n)$ is the smallest index $i$ for which $\alpha_{i,\delta} a(f_i |B_\delta, \mk n) \nl$ for some $\delta \mid N$ in \eqref{maineq} and once $i(\mk n)$ has been chosen, $\delta(\mk n)$ is the smallest divisor of $N$ such that $\alpha_{i(\mk n) ,\delta(\mk n) }  a(f_{i(\mk n)} |B_{\delta(\mk n)}, \mk n) \nl$. This is clearly an injective map, whence
\begin{equation}
\pi(f,x) \hookrightarrow \amalg_{i, \delta} \pi(f_i |B_\delta,   x) ,
\end{equation}
$\amalg$ being the \textsl{disjoint} union. We get 
\begin{equation}
\# \pi(f,x) \le \sum_{i, \delta} \pi(f_i |B_\delta,   x) \le \dim \mN \cdot \max_{i,\delta} \{ \# \pi(f_i |B_\delta,   x)  \} \ll \max_{i}  \{ \# \pi(f_i ,   x) \}
\end{equation}
and again results of Serre from \cite{serre}  (cf. \eqref{serreasy}) does the job.

Moreover if $N$ is square-free, the above argument also works almost verbatim for $\pi_{\mrm{sf}}(f,x)$. In this case we have the additional information: 

(i) $\Delta$ is square-free,

(ii) all the $\gamma_t$ appearing in \propref{mainprop} are all square-free rational numbers such that $(\gamma_t,N)=1$ (so that if $\mk n$ is square free and $(\mk n, \gamma_t N)=1$, $ \gamma_{t(\mk n)} \Delta \mk n$ is also square-free in \eqref{inj1}), and

(iii) Serre's asymptotic results clearly hold when we count over square-free integers as well. By this we mean an asymptotic formula of the form
\begin{equation}\label{biluasymp}
\pi_{\mrm{sf}, UN\ell }(\mk g, x) \sim c_f x/(\log x)^{\alpha(f)} \qq (\mk g \, \text{newform})
\end{equation}
for some constant $c_f$, $\alpha(f)$ as before.
This follows from the proof in \cite[p.~5]{serre}, in our case the generating function $\mc F(s)$ of $\pi_{\mrm{sf}, UN}(\mk g) := \{  n \text{ square-free },  (n,UN\ell)=1, \, a(\mk g,n) \nl   \}$ is just $\mc F(s)  = \sum_{n \in \pi_{\mrm{sf}, UN}(\mk g)} n^{-s}  = \prod_{p \in \pi_{\mrm{sf}, UN\ell}(\mk g)} (1+p^{-s})$ and the subsequent arguments in \cite{serre} hold verbatim, leading to the asymptotic formula \eqref{biluasymp}. 

These take care of the lower bound for $\pi_{\mrm{sf}}(f,x)$. The argument for the upper bound remains the same as before.
\end{proof}

\begin{exam} \label{delta-exam}
We work out an explicit set of primes $P$ outside of which one has  an asymptotic formula for $\pi(\Delta^2,x)$ which conforms to \propref{lubd}. We refer the reader to the LMFDB database for some of our calculations here. Namely $\dim S_{24}(1)=2$, and is spanned by two newforms say $X_1$ and $X_2$ 
\begin{equation}
X_1(\tau) = q + (540-\beta)q^2 +\cdots, \q  
X_2(\tau) = q + (540+\beta)q^2 +\cdots;
\end{equation}
which are conjugate under the Galois group $G\simeq \mf Z/2 \mf Z$ of the coefficient field $L=\mf Q(X_1) = \mf Q(X_2)= Q(\sqrt{D})$ with the fundamental discriminant $D=144169$. The Sturm's bound here is $\mc S(24,1)=2$ and the smallest prime $q$ such that $a(X_1,q) \ne a(X_2,q)$ is $q=2$. Following LMFDB, we put $\beta = 12 \sqrt{D}$. Finally note that $L$ has class number one, and that $D$ is a prime number.

We let $\mk l$ lie over $\ell$. One easily checks that (again by using Sturm's bound)
\begin{equation}
\Delta^2 = \frac{-1}{2\beta} X_1 + \frac{1}{2\beta}  X_2.
\end{equation}
Hence if we request that 
\begin{equation}
\mk l \nmid 2 \beta = 24 \sqrt{D} \iff \ell \nmid 2 \cdot 3 \cdot D,
\end{equation}
we would be able to write over $\mk l$-integral numbers
\begin{equation} \label{deltamodl}
\Delta^2 \equiv \frac{-1}{2\beta} X_1 + \frac{1}{2\beta}  X_2 \ml.
\end{equation}

This is all we need (of course here we got it much more directly, without using any analytic means like the Hecke bound) to obtain, by considering $T_2 (\Delta^2) - a(X_2, 2) \Delta^2$ that
\begin{equation}
a(X_1,n) \equiv a(\Delta^2, 2n) - a(\Delta^2, n/2) - a(X_2,2)a(\Delta^2,n) \ml.
\end{equation}
Then considering $n$ such that $(n, 6 D \ell)=1$, we find that
\begin{equation}
a(X_1,n) \equiv a(\Delta^2, 2n)  - (540+\beta)a(\Delta^2,n)  \ml,
\end{equation}
which shows by a simple counting that
\begin{equation} \label{ex1}
a(X_1,x) \le a(\Delta^2,2x).
\end{equation}
Moreover \eqref{deltamodl}  immediately implies (since $X_1,X_2$ are Galois conjugate) for any $j=1,2$,
\begin{equation} \label{ex2}
\pi(\Delta^2,x) \le \pi(X_j,x) .
\end{equation}
Combining \eqref{ex1} and \eqref{ex2}, we see that for $x \ge 2$ and $\ell \nmid 6 \cdot 144169$ (by Serre \cite{serre}) there is a constant $c(X_1, \ell)$ such that
\begin{equation}
\pi(X_1,x/2) \le \pi(\Delta^2,x) \le \pi(X_j,x) ,
\end{equation}
and thus for any $\epsilon>0$ and $x$ large enough
\begin{equation} \label{deltaex}
\frac{1}{2}(1-\epsilon)c(X_1, \ell) \cdot x/(\log x)^{\alpha(X_1)}     \le \pi(\Delta^2, 2  x) \le (1+ \epsilon) c(X_1, \ell) \cdot x/(\log x)^{\alpha(X_1)}    .
\end{equation}
This is in contrast to the result in \cite[\S~7.1.1]{js} where $\ell=3$ and an additional factor of $\log \log x$ was obtained in \eqref{deltaex}.
\end{exam}


\subsubsection{The quantity $h(f)$} \label{hf}
We now turn our attention to the quantity $h(f;\mk l)$ (denoted simply by $h(f)$ in the introduction), which is the exponent of $\log \log x$ appearing in \eqref{ellintro}. 
\begin{prop} \label{hfl}
Let $k,N$ be fixed and $\mk l$ be a prime in $K$ lying over an odd prime $\ell$. Suppose that for some $f \in \mN(\ok)$ which is non-singular $\ml$, one has $h(f; \mk l) \neq 0$. Then the following statements hold.

(i) $\mk l \mid \mc L(\{ q_{i,j} \};f)$ for \textsl{any} set of primes $\{q_{i,j}\}$ admissible for $f$.

(ii) There exists a pair $i,j \in S_f$ ($i \ne j$) such that the congruences $b_i(n) \equiv b_j(n) \ml$ hold for all $n \ge 1$ with $(n,N)=1$.
\end{prop}

\begin{rmk} \label{hflcomm}
In particular, $(i)$ says that the quantity $h_{f, \mk l} $ in \eqref{ellintro} is $0$ for all but finitely many $\ell$. Moreover, choosing the admissible set to be $\{ m_{i,j} \}$ from the introduction, and the bounds on $m_{i,j}$ from section~\ref{mij}, we see that $h(f;\mk l)=0$ for all $\ell > \mc C$, where $\mc C$ depends only on $k$ and $N$.
\end{rmk}

\begin{proof}
Let $f$ be as in the proposition and $\{q_{i,j}\}$ be admissible for $f$. If $\mk l \nmid \mc L(\{ q_{i,j} \};f)$
then the sieving procedure described in the proof of \propref{mainprop} works, and then \propref{ellasy} shows that $h(f, \mk l) = 0$. This proves $(i)$.

For $(ii)$, suppose that for all pairs $u \neq v$ ($u,v \in S_f$) one has $b_u(n) \not \equiv b_v(n) \ml$ for some $(n,N)=1$. By the multiplicativity and the Hecke relations, for each such pair $u \neq v$ there must by a prime, say $\mf q_{u,v}$ with $(\mf q_{u,v},N)=1$ such that $b_u(\mf q_{u,v}) \not \equiv b_v(\mf q_{u,v}) \ml$. Then clearly $\{ \mf q_{u,v} \}$ is an admissible set for $f$, and thus by the quantitative result from \propref{ellasy} we must have $h(f, \mk l) =0$. This contradiction finishes the proof of $(ii)$.
\end{proof}

We record here another feature of the quantity $h(f ;\mk l)$, whose proof is omitted since it is verbatim similar to that of \propref{ellasy}. Recall from \cite{jochno} that a set $\{\lambda(p)\}_{(p,N)=1}$ is called a system of eigenvalues $\ml$ if there is an eigenform $g \ml$ such that $T(p)g=\lambda(p)g$ for all $p$. Let $\widetilde{M}(N)$ be the $\ok/\mk l$-vector space consisting of reduction $\ml$ of elements of the algebra $\bigoplus_k \mnk(\ok)$.
\begin{prop} \label{conjod}
Let $\ell$ be an odd prime and $\mk l \mid \ell$. Let $f \in \widetilde{M}(N)$ be non-singular $\ml$. Suppose that $f$ can be written as a finite linear combination of eigenforms (of all Hecke operators) $\ml$ whose system of eigenvalues are pairwise distinct. Then $h(f;\mk l)=0$.
\end{prop}

This leads us to formulate the following conjecture.
\begin{conj} \label{conjs}
Let $\ell$ be an odd prime and $\mk l \mid \ell$. Let $f \in \widetilde{M}(N)$ be non-singular $\ml$. Then $h(f;\mk l)=0$ if and only if $f$ can be written as a finite linear combination of eigenforms (of all Hecke operators) $\ml$ whose system of eigenvalues are pairwise distinct.
\end{conj}
The conjecture, if true, implies via Jochnowitz's result \cite{jochno} on finitely many systems of eigenvalues $\bl$ that there are only finitely many $f \in \widetilde{M}(N)$ with $h(f; \mk l)=0$.

\section{Siegel modular forms} \label{siegelsec}
First we discuss a Sturm bound for Siegel modular forms with level. This would be required to quantify the congruence primes in our results to follow. For the Sturm bound, we follow some arguments by Ram Murty \cite{murty}.

\begin{prop} \label{siegel-sturm}
Let $F \in \mnk(\ok)$ be such that $F \nl$. Then there exist $T \in \Lambda_n$ with 
\begin{equation} \label{siestbd}
\mc M(T) \le \left(\frac{4}{3} \right)^n \frac{k}{16} [\spn \colon \Gamma_1(N)] =: \mc S^n(k,N)
\end{equation}
such that $a_F(T) \nl$.
\end{prop}

\begin{proof}
The proof is analogous to the argument used in \cite{murty} (essentially be reducing to level one \cite{raum} by using the `norm' map) and we do not repeat it. We just note that the main two ingredients that go into the proof in \cite{murty}:

(i) $\mnk$ has a basis consisting of elements with Fourier coefficients in $\mf Z$;

(ii) if $F \in \mnk(\mc O_K)$, then in each cusp, it's Fourier coefficients are also in a number field and have bounded denominators;

are available from the work of Shimura \cite{shimura}. 
\end{proof}

Next we recall the notion of a singular Siegel modular form of degree $n$. Namely, $F $ as \propref{siegel-sturm} is called singular $\ml$ if $a_F(T) \oml$ for all $T \in \Lambda_n^+$. We would not directly apply the following lemma in this paper, but indirectly to give a convenient hypothesis in \thmref{smodp}. Also we believe this is not written down in the literature, and could be useful elsewhere.

\begin{lem} \label{sing}
Assume that $N$ is coprime to $\ell$ ($\ell$ odd) and $F\in \mnk(\ok)$ and let ${\mathfrak l}$ be a prime of $K$ dividing $\ell$.
Assume that $F$ is $\mk l$ singular of rank $t$, i.e.
\[ r= \max\{ \mrm{rank}(T)\,\mid T\in \Lambda_n,\quad a_F(T)
\neq 0\bmod {\mathfrak l}\} \]   
with $1\leq t\leq n-1$.
Then $2k-t$ is divisible by $\ell-1$.
\end{lem}

This is just a $\Gamma_1(N)$ variant of Corollary 3.7 in \cite{boki}
so we only give a brief sketch of proof.

\begin{proof}
The basic strategy is to associate to $F$ as above an elliptic modular form $g$ for $\Gamma_1(R)$ of level
$R$ coprime to $p$ and congruent mod $\mathfrak l$ to a unit in ${\mathcal O}_{\mathfrak l}$;
the weight $k'$ of such a modular form must be divisible by $\ell-1$ by Serre et al.

To do this, we may
choose $T\in \Lambda_n$ of rank $t$ with 
$a_F(T)\neq 0\bmod {\mathfrak l}$ to be of the form $T=\begin{psm} 0 & 0\\ 0 & T_0\end{psm}$ with $T_0\in \Lambda_t^+$. Also, 
we may apply Siegel's $\Phi$-operator several times to go to a modular form $f$ of degree $t+1$.
Then we consider the Fourier-Jacobi coefficient
$\phi_{T_0}$ of $f$ and its theta decomposition; the modular form $h_0$ in this theta decompositon  is then
congruent to a constant
(unit) mod ${\mathfrak l}$.
Furthermore $h_0^2$ is a modular form for $\Gamma_1(N)\cap \Gamma_0(M)$
with nebentypus character $\left(\frac{-4}{*}\right)$ where
$M$ is the level of $2T_0$; the weight is $2k-t$.
If $M$ is coprime to $p$ we may take $g:=h_0^2$ and apply the 
statement from above.
Otherwise, we write $M=\ell^s\cdot M'$ with $M'$ coprime to $\ell$ and 
by standard level changing, there is a modular form $g$ for $\Gamma_1(NM')$ 
of weight $2k-t+ m\cdot (\ell-1)$ for some nonnegative integer $m$
with $g\equiv h_0^2\bmod {\mathfrak l}$; now we argue as before with this $g$.
\end{proof}

Let us recall now the following result due to B\"ocherer-Nagaoka, which was used to prove that congruences between Siegel modular forms imply congruences between their weights.
For $F \in \mnk(\ok)$ we put 
\[  \nu_{\mk l}(F):= \min_{T \in \Lambda_n} \nu_{\mk l}(a_F(T)).  \]

\begin{prop}[\cite{bona}] \label{bo-na}
Let $\mk l \subset \mc O_K$ be a prime ideal.
For every $F \in \mnk (\mc O_K)$ there exists for all sufficiently large $R \in \mf N$ an elliptic modular form $f \in M_k(\Gamma_1(NR^2))(\mc O_K)$ such that the Fourier coefficients of $f$ are finite sums of those of $F$; and $\nu_{\mk l}(f) = \nu_{\mk l}(F)$.
\end{prop}

For future use, let us briefly recall the setting of the proof of the above proposition, in a slight generality. Let $\mc P$ be a certain property satisfied by some of the Fourier coefficients of $F$.

We put $\mc M(T):= \max\{ T_{1,1}, \ldots,  T_{n,n} \}$. As in \cite{bona}, we consider the set
\begin{equation*}
\mc T =\{ T \in \Lambda_n | a_F(T) \text{ satisfies } \mc P    \}.
\end{equation*}
We let $d$ to be the minimum of the quantities $\mc M(T)$ for $T \in \mc T$. Then we fix any $T_0 \in \mc T$ with $\mc M(T_0)=d$ and put $diag(T_0)=(d_1,\ldots d_n)$. 

We now consider the finite set
\begin{equation} \label{td}
\mc T_d =\{ T \in \mc T | \mc M(T)=d    \}.
\end{equation}
We next choose $R \ge 1$ `large enough' (possibly with suitable additional conditions) such that
$\{ T \in \mc T_d | T \equiv T_0 \bmod R \} = \{T_0 \}$.

We then put, borrowing the notation from \cite{bona} 
\begin{equation} \label{Fr}
F^{(R,T_0)} (Z) := \sum_{T \equiv T_0 \bmod R} a_F(T) e(TZ),
\end{equation}
and define 
\begin{equation} \label{fdef}
f(\tau) = \sum_{r=0}^\infty a(f,r)q^r \in M^1_k(\Gamma_1(NR^2))
\end{equation}
by
\begin{equation} \label{maineq1}
 a(f,r) = \sum_{T \equiv T_0 \bmod R, \, diag(T) = (r, d_2,\ldots,d_n)} a_F(T) ;
\end{equation}
then one has
\begin{equation} \label{maineq2}
a(f,d_1) = a_F(T_0).
\end{equation}

We now show that an adaptation of this technique to our setting can be used to show that any non-zero $F$ as above has infinitely many $\glnz$-inequivalent `primitive' Fourier coefficients which are non-zero $\ml$. This generalises previous results on this topic by Zagier \cite{zag}, Yamana \cite{yam} where the Fourier coefficients were in $\mf C$.

We need a bit of more notation. For $T \in \Lambda_n^+$ ($n \ge 2$), let $\mc D(T)$ denote the greatest common divisor of all the diagonal elements of $T$ except the first:
\begin{equation}
\mc D(T) = \gcd(T_{2,2}, \ldots, T_{n,n}).
\end{equation}
Note that $\mc D(T) \in \mf N$ and that if $c(T)$ denotes the content of $T$, then $c(T) \mid \mc D(T)$.

\begin{thm} \label{ymodp}
Let $F$ as in \propref{bo-na}. Suppose there exist $T \in \Lambda_n$ such that $a_F(T) \nl$ with $(\mc D(T),N)=1$. Then there exist infinitely many $\glnz$-inequivalent primitive matrices $T \in \Lambda_n$ such that $a_F(T) \nl$ for all $\ell$ (lying under $\mk l$) effectively large enough in terms of only $k,N$; in particular we request $\ell>k+1$.
\end{thm}

\begin{proof}
We will freely refer to the discussion preceeding this proposition. Here the property $\mc P$ is that $a_F(T) \nl$ and $(\mc D(T),N)=1$. Fix a $T=\mf T$ with this property.

The next three paragraphs are meant to show how to effectively bound the integer $R$ from \eqref{Fr} in our situation. This would then tell us how large our $\ell$ has to be.
 
With such a $\mf T$ chosen, we now claim the existence of $T_0>0$ with $a_F(T_0) \nl$, $(\mc D(T_0), N)=1$ and $\mc M(T_0)=\max_i \{ d_i \} \le \mc S^n(k,N^2) $, where $\mc S^n(k,N)$ denotes the Hecke-Sturm bound for $\mnk$ and we have put $diag(T_0) = (d_1, d_2, \ldots, d_n)$. 

Suppose to the contrary. Taking $\mf{1}_N$ to be the trivial Dirichlet character $\bmod N$ and $L=L^t \in M_n(\mf Z)$ satisfying the conditions:
\[ L_{i,j} \equiv 0 \bmod N \, (i \ne j), \, L_{1,1} \equiv 0 \bmod N, \, \sum_{j=2}^n L_{j,j} d_j = \gcd(d_2 \cdots d_n)=\mc D(\mf T); \] 
we then consider the Fourier series
\begin{equation} \label{Gf}
G(Z) = \sum_T \mf{1}_N(\mrm{tr}(LT)) a_F(T)e(TZ).
\end{equation}
From Andrianov \cite{andria} we know that $G \in M^n_k(\Gamma)$, which is contained in $M^n_k(\Gamma_1(N^2))$ in view of the inclusions $\Gamma_1(N^2) \subset \Gamma \subset \Gamma_1(N)$, with $\Gamma = \{\gamma =  \begin{psm} A&B \\ C&D \end{psm} \in \Gamma_1(N) | C \equiv 0 \bmod N^2  \}$. Further our hypothesis on $F$ shows that $ G \nl$ (indeed $a_G(\mf T) \nl$). 

By using the Sturm bound for $\Gamma^n_1(N^2)$, we get the existence of $T_0$ as claimed above, because the Fourier expansion of $G$ in \eqref{Gf} is supported only on those $T$ for which $(\mc D(T),N)=1$. 

We work with this $T_0$ as per the set-up described following \propref{bo-na} and also assume (without loss) that $\mc M(T_0)=d$ is the minimum with respect to the property $\mc P$. We then consider $\mc T_d$ as in \eqref{td} and proceed to choose $R$ suitably.

We require three properties of such an $R$: it should be effectively bounded in terms of $k,N$; should be coprime to $ \mc D(T_0)$; and should be big enough so that the set $\{ T \in \mc T_d | T \equiv T_0 \bmod R , (\mc D(T),N)=1\} = \{T_0 \}$. We note that the following choice $R = ([2 \mc S(k,N^2)] +1) \cdot  \mc D(T_0) +1$is good. This will ensure that one of the diagonal congruences for $ T \equiv T_0 \bmod R$ does not hold in view of the definition of $\mc T_d$; further this choice also ensures that $(R,\mc D(T_0))=1$.

Next, we consider the modular form $f \in \Gamma_1(NR^2)$ from \eqref{fdef} with Fourier expansion as in \eqref{maineq1} obtained from our $F$. Recall that by
construction $\nu_{\mk l}(f) = 0$, and we consider the Fourier coefficients $a(f,r)$ of $f$. Since $\ell>k+1$, $f$ is not singular $\ml$.
From \eqref{maineq2} it follows that if $a(f,r) \nl$, there must exist $T \equiv T_0 \bmod R$
with $a_F (T) \nl$ such that $diag(T) = (r, d_2, . . . , d_n) = (r,\mc D(T_0)) $. Since $(\mc D(T_0), NR^2)=1$ by our construction, we can apply \thmref{emf} with $Q:= \mc D(T_0)$ (note that the level of $f$ divides $NR^2$) to deduce that there must exist infinitely many $r$ such that $a_f(r) \nl$ such that $(r, \mc D(T_0))=1$. 

The crucial point is that for all the Fourier coefficients $a_F(T)$ ($T \equiv T_0 \bmod R$) of $F^{(R,T_0)}$  from \eqref{Fr}, one has $\mc D(T)=\mc D(T_0)$ ($= (d_2, \ldots,d_n)$). Therefore we get the existence a sequence of infinitely many matrices $T$ such that $a_F(T) \nl$ with the property that their diagonal entries have $gcd$ to be $1$. This implies that the $T$ under consideration are all primitive.

Unfortunately this \textsl{does not} imply that all the primitive $T$ obtained in this way are inequivalent under the action of $\glnz$. However if $T \in \Lambda_n^+$, we can directly invoke \propref{primfin}, whose proof however is deferred until the end of this proof, to get hold of infinitely many such matrices which are pairwise distinct $\bmod\, \glnz$. Otherwise if $\mrm{rank}(T)=s <n$, we can find $U \in \glnz$ such that $T[U] = \begin{psm} \widetilde{T} & 0 \\ 0 & 0 \end{psm}$. Then we can consider $G(Z):= \Phi^{n-s}(F)(Z) = \sum_{S \in \Lambda_s} a_F(\begin{psm} S & 0 \\ 0 & 0 \end{psm}) e(SZ)$, ($Z \in \mf H_s$). Clearly $G \nl$ and in particular $a_G(\widetilde{T}) \nl$ with $\widetilde{T} \in \Lambda_s^+$ primitive. Thus we can again invoke \propref{primfin} to $G$ to conclude the existence of infinitely many primitive $S \in \Lambda_s^+ / \mrm{GL}(s, \mf Z)$ such that $a_G(S) \nl$. To conclude the same result for $F$, note that the matrices $\begin{psm} S & 0 \\ 0 & 0 \end{psm}$ obtained above as also pairwise distinct $\bmod\, \glnz$. This follows from the statement that $S_1,S_2 \in \Lambda_s^+$ are $\mrm{GL}(s, \mf Z)$ equivalent if and only if $\begin{psm} S_1 & 0 \\ 0 & 0 \end{psm}$ and $\begin{psm} S_2 & 0 \\ 0 & 0 \end{psm}$ are $\glnz$ equivalent. One side of the implication is trivial, to see the other; suppose that $U= \begin{psm} U_1 & U_2 \\ U_3 & U_4 \end{psm} \in \glnz$ be such that
\[    \begin{pmatrix} U^t_1 & U^t_3 \\ U^t_2 & U^t_4 \end{pmatrix}  \begin{pmatrix} S_1 & 0 \\ 0 & 0 \end{pmatrix} \begin{pmatrix}
U_1 & U_2 \\ U_3 & U_4 \end{pmatrix}   =  \begin{pmatrix} S_2 & 0 \\ 0 & 0 \end{pmatrix} .  \]
A short calculation shows that $U_1^tS_1U_1 = S_2$ and $U_2^t S_1U_2=0$. The positive-definiteness of $S_1$ forces $U_2=0$ and thus $U_1 \in \mrm{GL}(s, \mf Z)$. Since the content is preserved under the action of $\glnz$, we are therefore done.

While applying \thmref{emf} we needed to ensure that $\ell $ is large enough only in terms of $k,N,R$. But since $\mc D(T_0)$, and hence $R$ (see our choice of $R$) can be estimated (explicitly) as a polynomial in $k,N$, $\ell$ is large enough depending only on $k,N$. We do not work this out.
\end{proof}

It remains to prove \propref{primfin}. Let us recall the Fourier-Jacobi expansion of $F$ of type $(1,n-1)$ from \eqref{fj}. Let $\phi_S$ be the Fourier-Jacobi coefficients. By tacitly identifying $\phi_S$ with the associated function $\varphi_S = \phi_S \cdot e(SZ)$ (cf. section~\ref{prelim}), we refer to the Fourier coefficients of $\phi_S$ as supported on matrices of the form $\begin{psm} * & * \\ * & S \end{psm} \Lambda_n$.
From \eqref{jacobi2} the Fourier expansion of the theta components of 
\begin{equation} \label{hrecap}
 h_{\mu_0}(\tau)=\sum_r b(r,\mu_0) e (r-S^{-1}[\mu_0/2])\cdot \tau). 
\end{equation} 

\begin{lem} \label{jacprim}
With the above notation, suppose that for all $S \in \Lambda^+_{n-1}$ and all $\mu \in \mf Z^{n-1}/ 2S \mf Z^{n-1}$ we have $b(r,\mu)\nl$ only for finitely many values of
$r-S^{-1}[\mu]$ with $r$ coprime to $\gamma(\mu)$, where $\gamma(\mu):=  \gcd (\mu, c(S) )$. Then the Fourier expansion of $\phi_S \ml$ is supported on matrices $T = \begin{psm} r & \mu^t \\ \mu & S \end{psm} \in {\Lambda_n}$,
with $\mrm{rank}(T)\leq n-1.$
\end{lem}

\begin{proof}
We start from the theta expansion \eqref{jacobi1} of $\phi_S$.
We observe that for a Fourier series as in \eqref{jacobi2} the invariance property \eqref{jacobi3}
is equivalent to the possibility writing down a "theta expansion" as in (\ref{jacobi1}). This observation will be used soon.

We put $b(r,\mu)^*:= b(r,\mu)$ if the gcd of $r$, $\mu$ and of the content $c(S)$ is one and we define 
$b(r,\mu)$ 
to be zero otherwise.
We observe that the gcd of $r$, $\mu$ and of $c(S)$  is the same
as the gcd
of $r+L^t\cdot \mu+S[L]$, $\mu+2S\cdot L$ and of   $c(S)$.
This implies that the subseries of $\phi_S$, defined by 
\[ \phi_S^{*}(\tau,z)= \sum_{r,\mu} b(r,\mu)^* e(r\tau +\mu^t z) \]
still has an expansion (keeping in mind the observation from the first paragraph)
\[ \phi^*_S( \tau,z)=\sum_{\mu_0} h^{*}_{\mu_0}(\tau)\Theta_S[\mu_0](\tau,z) \]
where the $h^{*}_{\mu_0}$ are given by subseries of the Fourier expansion of $h_{\mu_0}$,
more precisly, we have
\[
h^*_{\mu_0}(\tau)=\sum_r b(r,\mu_0)^*e(r-S^{-1}[\frac{\mu_0}{2}])\cdot \tau).
\]
For fixed $\mu_0$ we may rephrase the condition defining $b(r,\mu_0)^*$  as
saying that the gcd of $ r$ and $\gamma(\mu_0)$ is one where 
$\gamma(\mu):=$  gcd of $\mu$ and of  $c(S) $.   

This can be rephrased in terms of the summation index
$r-S^{-1}[\frac{\mu_0}{2}]$ by
$$r-S^{-1}[\frac{\mu_0}{2}]\notin -S^{-1}[\frac{\mu_0}{2}] +q\cdot {\mathbb Z}$$
for any prime $q$ dividing $\gamma(\mu_0)$. In particular, $h_{\mu_0}^*$ is still a modular form
of some level  because it is extracted from the modular form $h_{\mu_0}$ by some coprimality condition
for its summation index .

By hypothesis, for all $\mu_0$ that  $b(r,\mu_0)\ne 0\bmod p$ only for finitely many 
$r-S^{-1}[\mu_0]$ with $r$ coprime to $\gamma(\mu_0)$. 
Then by  \cite{serre1} or Theorem 3.1 in \cite{boki} the $h^{*}_{\mu_0}$ are constant functions $\ml$:  
\[\phi_S^{*}\equiv  \sum b(r,\mu_0)\Theta_S[\mu_0](\tau,z)   \ml \]
where  $r$ and $\mu_0$ satisfy $r=S^{-1}[\frac{\mu_0}{2}]$ (see \eqref{hrecap}) and $(r,\gamma(\mu_0))=1$. In particular, the Fourier expansion of $\phi_S^{*} \ml$ is supported on matrices $T\in {\Lambda^n}$,
with $\mrm{rank}(T)\leq n-1$.
\end{proof}

\begin{prop} \label{primfin}
If $F\in M_k^n(\Gamma_1(N))$, be such that $F \nl$; then the set
\[
\{T\in \Lambda^+_n\,\mid T \text{  primitive  }, a_F(T) \nl\}
/\glnz
\]
is either empty or infinite.
\end{prop}
\begin{rmk}
The same statement, but without primitivity condition appears in \cite{boki}.
\end{rmk}

\begin{proof}
We assume that the set in question is indeed non-empty and finite with
$\{L_1\dots L_t\}$ ($t \ge 1$) as a set of representatives.
Let $\phi_S$ be the Fourier-Jacobi coefficients of $F$ of type $(1,n-1)$.
Then $\Phi_S$ can carry a $\ml$ nonzero primitive Fourier coefficient only
if $S=L_i[G]$ for some primitive $G\in {\mf Z}^{(n,n-1)}$.
 Furthermore, for fixed $S$ we observe that $ \det L_i = \det(S) \cdot (r-S^{-1}[\mu])$
and hence $r-S^{-1}[\mu_0]$ is from a finite set. Since $L_i$ is primitive, necessarily $\gcd(r, \gamma(\mu))=1$, where recall that $\gamma(\mu) = \gcd(\mu^t,c(S))$.
We may now apply \lemref{jacprim} to see that the rank condition for the $L_i$ cannot be satisfied. This contradiction finishes the proof.
\end{proof}

\begin{rmk}
The reasoning above also works in the archimedean setting.
\end{rmk}

\subsection{Quantitative results: Siegel modular forms} \label{siegquanta}
In this subsection, we collect various quantitative results on the number of nonvanishing Fourier coefficients $\ml$ of Siegel modular forms, which essentially follow from the corresponding statements about elliptic modular forms. Let $F  \in \mnk (\mc O_K)$, and define the sets
\begin{align}
\pi_F(x, \det) &:= \{ d \le x |   a_F(T) \nl \text{ for some } T \in \Lambda^+_n \text{ such that } \det(T) =d \} \\
\pi_F(x, \det; \mrm{sf}) &:= S_F(x, \det) \cap \{ \text{Odd square-free numbers} \}\\
\pi_F(x, \det; \mrm{pr}) &:= S_F(x, \det) \cap \{ \text{Odd prime numbers} \}\\
\pi_F(x, \tr) & :=
\{ d \le x |   a_F(T) \nl \text{ for some } T \in \Lambda^+_n \text{ such that } \tr(T) =d  \}
\end{align}

\begin{thm} \label{smodp}
Let $F \in \mnk$ be such that $F$ be non-singular $\ml$. Then for some $0< \beta(F) \le 3/4$,
\begin{align}
&(a1) \q | \pi_F(x, \det)| \gg_F x/(\log x)^{\beta(F)}  \q (n \text{ odd} ), \\
&(a2) \q | \pi_F(x, \det)| \gg_F \sqrt{x}/(\log \log x)  \q (n \text{ even}  ), \\
& (b)  \q | \pi_F(x, \det; \mrm{sf}) | \gg  x/(\log x)^{\beta(F)} \q (n \text{ odd, } N=1, k \ge (n+3)/2,\ell \gg_{F} 1)\\
& (c)  \q | \pi_F(x, \det; \mrm{pr}) | \gg  x/(\log x) \q (n \text{ odd, } N=1, k \ge (n+3)/2).
\end{align} 
\end{thm}
In view of \lemref{sing}, the above theorem therefore holds for all suitable $\ell$ such that $\ell -1 \ge k-n$. 

\begin{proof}
Since $F$ is non-singular $\ml$, we first get hold of $T \in \Lambda^+_n$ such that $a_F(T) \nl$, say for concreteness that $\det(T)$ is minimal with this property. 
Let $T_0$ be the right lower diagonal block of $T$ of size $n-1$. We look at the $T_0$-th Fourier-Jacobi coefficient, say $\phi$, of $F$. We consider any of it's theta components, say $h_\mu$ ($\mu \in \mf Z^{n-1}/2T_0\cdot \mf Z^{n-1}$), which is $\nl$. Such a $h_\mu$ exists since $\phi \nl$. It is well-known that $H(\tau) :=h_\mu( 4d_{T_0} \tau)$ is in $\mm$ with $M = 16  d_{T_0}^2N$ and $\kappa:= k - (n-1)/2$. We note that the Fourier coefficients of $H$ and $F$ are related as (for this and the above facts, see e.g., \cite[section~2.3]{bo-da}):
\begin{equation} \label{fh}
a(H, n) = a_F( \begin{pmatrix}
 n /d_{T_0} +  T_0^{-1}[\mu/2] & \mu/2 \\ \mu^t/2 & T_0   \end{pmatrix} ), \q ( {n / d_{T_0}} +  T_0^{-1}[\mu/2] \in \mf N).
\end{equation} 
We now have two avenues. Since $T_0$ is fixed, we can get (in an elementary way) $(a1)$ from \eqref{fh} by applying \propref{ellasy} to $H$ (which is non-singular $\ml$) and taking $\beta(F):= \alpha(H)$ if $n$ is odd. 
However we use the stronger result in \cite[Theorem~1]{js} which holds for all $\ell$. Since $n$ is even in $(a2)$, we apply \cite[Theorem~1]{jsh} to $H$ and are done.

For $(b)$, by combining the results in Proposition~3.8 and Theorem~1.1 of \cite{bo-da} we choose $T = \begin{psm}  n & r/2 \\ r^t/2 & T_0 \end{psm} \in \Lambda^+_n$ to be such that $a_F(T) \ne 0$, $d_T, d_{T_0}$ are odd, square-free and $r$ is `$T_0$-primitive', i.e., the denominator of $T_0^{-1}[r/4]$ is exactly $d_{T_0}$. See \cite[Prop.3.8]{bo-da}. We then choose $\ell$ such that $\ell \nmid a_F(T)$ by requesting $\ell \gg_F 1$ to be large enough, e.g., by using Hecke's bound 

Arguing as in the previous paragraph we get $H$ of level $\Gamma_1(d^2_{T_0})$ such that $H \nl $ and $H$ is related to $F$ by \eqref{fh}. Moreover from \cite{bo-da}, $a(H,n) \ne 0$ only if $(n, d_{T_0})=1$.

If there exists $n \ge 1$ square-free such that $a(H,n) \nl$, then we are done by \cite[Theorem~26]{js}. Otherwise, from the above paragraph and from the condition in (the second half of) \cite[Theorem~26]{js}, we conclude that $a(H,n) \not \equiv 0 \ml$ only for those $n$ which are divisible by $v^2$ with $v \mid d_{T_0} \ell$ for some prime $v$. $v$ cannot divide $d_{T_0}$ as the Fourier expansion of $H$ is supported away from $d_{T_0}$. Thus $v=\ell$ and $(\ell, d_{T_0})=1$.
From \thmref{emf}~$(a)$ (again using that $\ell$ is large enough) we would then have a contradiction unless $H \equiv 0 \ml$. The lower bounds follow from those of $H$, upon using \eqref{fh}.

Finally arguing exactly as in the proof of $(b)$ above, $(c)$ follows from \cite[Thm.~I]{j}. This completes the proof.
\end{proof}

\begin{rmk}
It is not clear whether the arguments of \cite{jsh} can be adapted to deal with square-free Fourier coefficients. In this regard, the method of this paper also does not work for half-integral weight modular forms, since the Hecke operators here are indexed by squares. If one had such a result, then part $(b)$ above would have a version for even $n$ as well.
\end{rmk}

\begin{rmk} \label{lexpl}
The lower bound on $\ell$ in~$(b)$ can be made more explicit if we know bounds on the smallest $d_T$ (say of the form $d_T \ll k^{a_n}N^{b_n}$) for the fundamental $T$ such that $a_F(T) \ne 0$ \textsl{for all} $n$. This seems known only for $n=1$ from \cite{ad2}. 
\end{rmk}

\begin{rmk}
In light of \thmref{smodp}~$(b)$, it may seem that \propref{ymodp} is redundant; but the former is only for odd $n$, whereas the latter is for all $n$. Moreover the lower bound on $\ell$ in \propref{ymodp} is effective only in terms of the weight and level of the concerned space, whereas in the case of \thmref{smodp}~$(b), $ it is dependent on the modular form.
\end{rmk}

We end by noting a nonvanishing result in terms of the trace function.
\begin{prop} \label{strace}
Let $F \in \mnk$ such that $F \nl$. Then for all primes $\mk l$,
\begin{equation}
| \pi_F(x, \tr) | \gg x/(\log x)^{\beta(F)}.
\end{equation}
\end{prop}

\begin{proof}
We apply the procedure discussed after \thmref{ymodp} to $F$. Here the property $\mc P$ is  is that $a_F(T) \nl$. The proof then follows trivially from \eqref{maineq1}, \eqref{maineq2} and \cite[Thm.~1]{js}. If one wishes to settle for more simple minded proof, then \propref{ellasy} may be used, but at the price of $\ell$ being large. In any case we put $\beta(F):= \alpha(f)$, where $f$ is as given in \eqref{fdef}.
\end{proof}

\end{document}